\documentclass[12pt]{amsart} 
\usepackage{amssymb,amsmath,amsthm,amscd}
\usepackage[all]{xy}
\usepackage{amsfonts}
\usepackage{float}
\usepackage{enumitem}

\usepackage{graphicx}

\usepackage{hyperref}

\numberwithin{equation}{section}

\newtheoremstyle{fancy1}{10pt}{10pt}{\itshape}{12pt}{\textsc\bgroup}{.\egroup}{8pt}{ }
\newtheoremstyle{fancy2}{10pt}{10pt}{}{12pt}{\itshape}{.}{8pt}{ }

\theoremstyle{fancy1}

\newtheorem{lem}[equation]{lemma}
\newtheorem{prop}[equation]{Proposition}
\newtheorem{theorem}[equation]{Theorem}

\newtheorem{main}{Theorem}

\newtheorem*{main*}{Theorem}
\newtheorem*{conj*}{Conjecture}
\newtheorem*{cor*}{corollary}

\newtheorem*{question*}{Question}
\newtheorem{conj}{Conjecture}
\newtheorem*{cor1}{Corollary 1}

\newtheorem*{problem*}{Problem}

\theoremstyle{fancy2}
\newtheorem{definition}[equation]{Definition}

\newtheorem*{rem*}{Remark}

\newtheorem*{rems*}{Remarks}

\newcommand{\cref}[1]{Corollary~\ref{#1}}

\newcommand{\lref}[1]{Lemma~\ref{#1}}
\newcommand{\pref}[1]{Proposition~\ref{#1}}

\newcommand{\tref}[1]{Theorem~\ref{#1}}
\newcommand{\sref}[1]{Section~\ref{#1}}

\newcommand{\pw}{\partial W}

\newcommand{\e}{\epsilon}

\newcommand{\ggm}{geometric graph manifold\ }

\newcommand{\lccs}{local connected components }
\newcommand{\ggms}{geometric graph manifolds\ }

\newcommand{\Sph}{\mathbb{S}}

\newcommand{\su}{L^2}
\newcommand{\D}{L}

\newcommand{\gencyl}{(\su\times\R^{n-2})/G}
\newcommand{\gencylrec}{\su\times\R^{n-2}}

\newcommand{\cyl}{(\su\times\R^{n-2})/G}

\newcommand{\R}{{\mathbb{R}}}
\newcommand{\Z}{{\mathbb{Z}}}

\newcommand{\N}{{\mathbb{N}}}

\def\con#1=#2(#3){#1 \equiv #2 \bmod{#3}}

\newcommand{\la}{\langle}
\newcommand{\ra}{\rangle}

\newcommand{\tr}{\ensuremath{\operatorname{tr}}}

\renewcommand{\Im}{\ensuremath{\operatorname{Im}}}
\newcommand{\im}{\ensuremath{\operatorname{Im}}}

\newcommand{\vol}{\ensuremath{\operatorname{vol}}}

 \DeclareMathOperator{\Iso}{Iso}

\DeclareMathOperator{\Scal}{Scal}

\newcommand{\spa}{\mbox{span}}

\begin{document}

\title{Manifolds with conullity at most two\\as graph manifolds} 

\author{Luis A. Florit}
\address{IMPA: Est. Dona Castorina 110, 22460-320, Rio de Janeiro,
Brazil}
\email{luis@impa.br}
\author{Wolfgang Ziller}
\address{University of Pennsylvania: Philadelphia, PA 19104, USA}
\email{wziller@math.upenn.edu}
\thanks{The first author was supported by CNPq-Brazil,
and the second author by a grant from the National Science
Foundation, by IMPA, and CAPES-Brazil}

\begin{abstract} 
We find necessary and sufficient conditions for a complete Riemannian
manifold $M^n$ of finite volume, whose curvature tensor has nullity at
least $n-2$, to be a \ggm\!\!. In the process, we show that Nomizu's
conjecture, well known to be false in general, is true for manifolds with
finite volume.
\end{abstract}
\maketitle

The nullity space $\Gamma$ of the curvature tensor $R$ of a Riemannian
manifold $M^n$ is defined for each $p\in M$ as $\Gamma(p)=\{X\in T_pM:
R(X,Y)=0\ \ \forall Y\in T_pM\}$, and its dimension $\mu(p)$ is called
the {\it nullity} of $M^n$ at $p$. It is well known that the existence of
points with positive nullity has strong geometric implications. For
example, on an open subset of $M^n$ where~$\mu$ is
constant, $\Gamma$ is an integrable distribution with totally geodesic
leaves. In addition, if~$M^n$ is complete, its leaves are also complete
on the open subset where $\mu$ is minimal; see e.g.~\cite{ma}.
Riemannian $n$-manifolds with conullity at most 2, i.e., $\mu\geq n-2$,
which we call {\it CN2 manifolds} for short, appear naturally and
frequently in several different contexts in Riemannian geometry, e.g.:

\begin{itemize}[leftmargin=0.6cm]
\item Gromov's 3-dimensional graph manifolds admit a complete CN2 metric
with nonpositive sectional curvature and finite volume whose set of flat
points consists of a disjoint union of flat totally geodesic tori
(\cite{g}). These were the first examples of Riemannian manifolds with
geometric rank one. Interestingly, any complete metric of nonpositive
curvature on such a graph manifold is necessarily CN2 and quite rigid, as
was shown in \cite{sc};

\item A Riemannian manifold is called semi-symmetric if at each point the
curvature tensor is orthogonally equivalent to the curvature tensor of
some symmetric space, which is allowed to depend on the point. CN2
manifolds are semi-symmetric since they have pointwise the curvature
tensor of an isometric product of a Euclidean space and a surface with
constant curvature. Conversely, Szab\'o showed in \cite{sz} that a
complete simply connected semi-symmetric space is isometric to a
Riemannian product $S\times N$, where $S$ is a symmetric space and $N$
is, on an open and dense subset, locally a product of CN2 manifolds;

\item Isometrically deformable submanifolds tend to have large nullity.
In particular, by the classic Beez-Killing theorem, any locally
deformable hypersurface in a space form has to be CN2. Yet, generically,
CN2 hypersurfaces are locally rigid, and the classification of the
deformable ones has been carried out a century ago in \cite{sbr,ca}; see
\cite{dft} for a modern version and further results. The corresponding
classification of locally deformable CN2 Euclidean submanifolds in
codimension two is considerably more involved, and was obtained only
recently in \cite{df} and \cite{ff};

\item A compact immersed submanifold $M^3\subset\R^5$ with nonnegative
sectional curvature not diffeomorphic to the 3-sphere $\Sph^3$ is
necessarily CN2, and either isometric to $(\Sph^2\times\R)/\Z$ for some
metric of nonnegative curvature on $\Sph^2$, or diffeomorphic to a lens
space $\Sph^3/\Z_p$; see \cite{fz}. In the case of lens spaces, the set
of points with vanishing curvature has to be nonempty with Hausdorff
dimension at least two. However, it is not known yet if they can be
isometrically immersed into $\R^5$;

\item I. M. Singer asked in \cite{si} whether a Riemannian manifold is
homogeneous if the curvature tensor at any two points is orthogonally
equivalent. The first counterexamples to this question were CN2
manifolds with constant scalar curvature, which clearly have this
property, and are typically not homogeneous; see \cite{se1,BKV}.

\end{itemize}

\medskip

The most trivial class of CN2 manifolds is given by cylinders
$\gencylrec$ with their natural product metrics, where $\su$ is
any (not necessarily complete) connected surface. More generally, we call
a {\it twisted cylinder} any quotient
$$
C^n=\gencyl,
$$
where $G\subset\Iso(\gencylrec)$ acts properly discontinuously and
freely. The natural quotient metric is clearly CN2, and we call $\su$ the
{\it generating surface} of $C^n$, and the images of the Euclidean factor
its {\it nullity leaves}. Observe that $C^n$ fails to be complete only
because $\su$ does not need to be. Yet, what is important for us is that
$C^n$ is foliated by complete, flat, totally geodesic, and locally
parallel leaves of codimension~$2$.

\medskip

Our first goal is to show that these are the basic
building blocks of complete CN2 manifolds with finite volume:

\begin{main}\label{m1}
Let $M^n$ be a complete CN2 manifold. Then each finite volume connected
component of the set of nonflat points of $M^n$ is globally isometric to
a twisted cylinder.
\end{main}

The hypothesis on the volume of $M^n$ is essential, since complete
locally irreducible Riemannian manifolds with constant conullity two
abound in any dimension; see \cite{se1}, \cite{BKV} and references
therein. These examples serve also as counterexamples to the Nomizu
conjecture in \cite{no}, which states that a complete locally irreducible
semi-symmetric space of dimension at least three must be locally
symmetric. However, \tref{m1} together with Theorem 4.4 in \cite{sz}
yield:

\begin{cor1}
Nomizu's conjecture is true for manifolds with finite volume.
\end{cor1}

For the 3-dimensional case, the fact that the set of nonflat points of a
finite volume CN2 manifold is locally reducible was proved in \cite{se2}
and \cite{sw} with a longer and more delicate proof; see also \cite{se3}
for the 4-dimensional case. Notice also that in dimension 3
the CN2 condition is equivalent to the assumption, called cvc(0) in
\cite{sw}, that every tangent vector is contained in a flat plane,
or to the condition that the Ricci endomorphism has eigenvalues
$(\lambda,\lambda,0)$.
Furthermore, in \cite{sb} it was shown that a complete 3-manifold with
(geometric) rank one is a twisted cylinder.

\medskip

Observe that we are free to change the metric in the interior of the
generating surfaces of the twisted cylinders in \tref{m1}, still
obtaining a complete CN2 manifold. Moreover, they are nowhere flat with
Gaussian curvature vanishing at their boundaries. Of course, these
boundaries can be quite complicated and irregular.

\medskip

In general it is very difficult to fully understand how the twisted
cylinders in \tref{m1} can be glued together through the set of flat
points in order to build a complete Riemannian manifold. An obvious way
of gluing them is through compact totally geodesic flat hypersurfaces.
Indeed, when the boundary of each generating surface $\su$ in the twisted
cylinder $C=\gencyl$ is a disjoint union of complete geodesics $\gamma_j$
along which the Gaussian curvature of $\su$ vanishes to infinity
order, the boundary of $C$ is a disjoint union of complete totally
geodesic flat hypersurfaces
$H_j=(\gamma_j\times\R^{n-2})/G_j\subset M^n$, where $G_j$ is the
subgroup of~$G$ preserving $\gamma_j$. We can now use each $H_j$ to
attach another finite volume twisted cylinder $C'$ to $C$ along $H_j$, as
long as $C'$ has a boundary component isometric to~$H_j$. Repeating and
iterating this procedure with each boundary component we construct a
complete CN2 manifold $M^n$. As we will see, the hypersurfaces $H_j$ have
to be compact if $M^n$ has finite volume. This motivates the
following concept of \ggm of dimension $n\geq 3$, which by
definition is endowed with a CN2 Riemannian metric:

\medskip

{\it Definition.}
A connected Riemannian manifold $M^n$ is called a {\it \ggm}
if $M^n$ is a locally finite disjoint union of twisted cylinders $C_i$
glued together through disjoint compact totally geodesic flat
hypersurfaces $H_\lambda$ of $M^n$. That is,
$$
M^n\setminus W= \bigsqcup_\lambda H_\lambda,
\ \ \ {\rm where}\ \ \ W:=\bigsqcup_i C_i.
$$
\begin{figure}[!ht]
\centering
\includegraphics[width=0.74\textwidth]{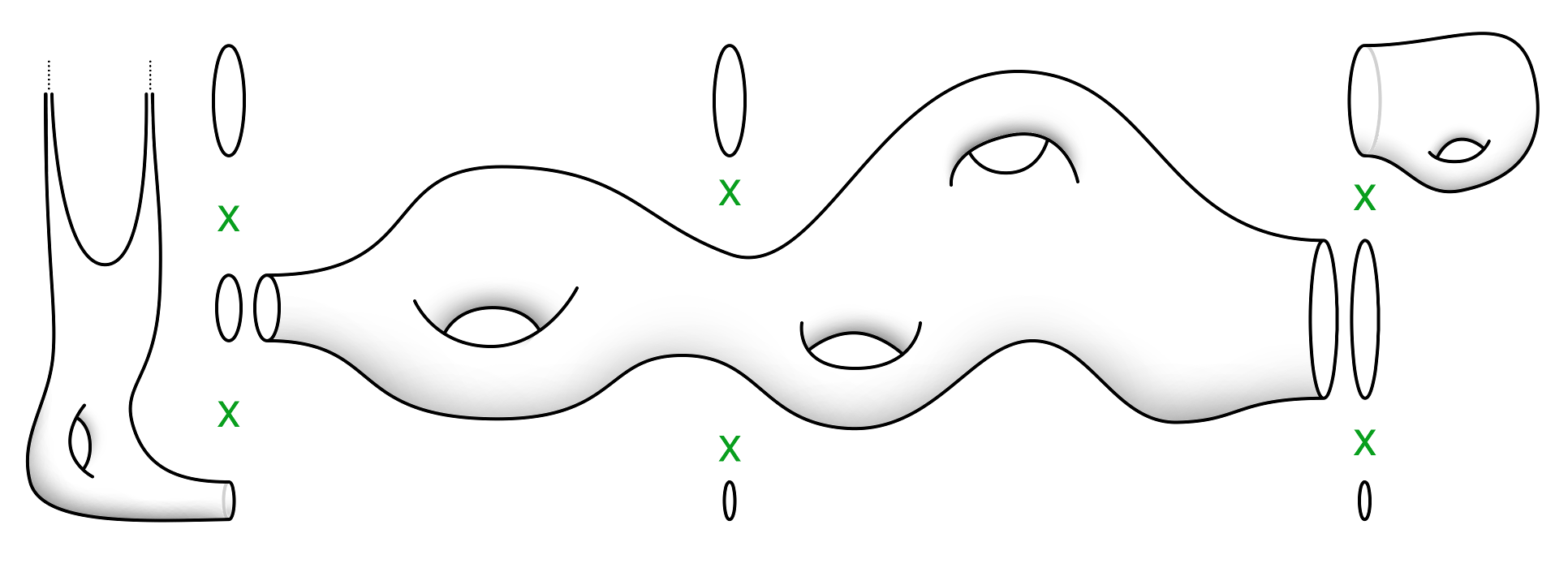} 
\caption{\small An irreducible 4-dimensional CN2 \ggm}
\centerline{\small with three cylinders and two (finite volume) ends}
\end{figure}
\bigskip

Here we allow the possibility that a hypersurface $H_\lambda$ is
one-sided, even when $M^n$ is orientable. We also assume, without loss of
generality, that the nullity leaves of two cylinders $C$ and $C'$, glued
along~$H_\lambda$, have distinct limits in $H_\lambda$. This implies in
particular that for each cylinder $C$, the Gauss curvature vanishes along
$\partial C$ to infinite order. Notice that the locally finiteness
condition is equivalent to the assumption that each $H_\lambda$ is a
common boundary component of two twisted cylinders $C_i$ and $C_j$, that
may even be globally the same, each lying at a local side of $H_\lambda$.

Observe that the complement of $W$ is contained in the set of flat points
of $M^n$, but we do not require that the generating surfaces of $C_i$
have nonvanishing Gaussian curvature. In particular the sectional
curvature of $C_i$, or equivalently its scalar curvature, can change
sign. More importantly, $W$ carries a well defined complete flat totally
geodesic parallel distribution of constant rank $n-2$ contained in the
nullity of $M^n$. Furthermore, $W$ is dense and locally finite in the
sense that it has a locally finite number of connected components (see
\sref{mainsec} for a precise definition). These two topological
properties will be crucial in what follows, so for convenience we say
that a dense locally finite set is {\it full}.

\hspace{0.5cm}

A natural way to try to see if a Riemannian manifold $M^n$ as in
\tref{m1} is indeed a \ggm is the following. \tref{m1} implies that on
the open set of nonflat points $V$ we have the well defined parallel
nullity distribution $\Gamma$ of rank $n-2$, as in~$W$ above. Now,
consider any open set $\hat V\supset V$ carrying a complete flat totally
geodesic distribution $\hat \Gamma$ with $\hat\Gamma|_V=\Gamma$, which we
call an {\it extension} of $V$. We will show that each connected
component of $\hat V$ is still a twisted cylinder, and call $\hat V$
{\it maximal} if it has no larger extension. Clearly, by definition $V$
always has a maximal extension, but it may not be unique. More
importantly, all extensions of $V$ may fail to be either dense, or
locally finite, or both; see Examples 2 and 3 in Section 1.

Our second main goal is to prove that all we need to ask in order for
$M^n$ as in \tref{m1} to be a \ggm is that some extension of $V$ is full:

\begin{main}\label{m2}
Let $M^n$ be a complete CN2 manifold with finite volume. Then $M^n$ is a
\ggm if and only if its set of nonflat points $V$
admits a full extension. In particular, if $V$ itself is full, then $M^n$
is a \ggm\!\!.
\end{main}

We point out that here we do not require a full extension $\hat V$ of
$V$ to be maximal, but clearly any maximal extension of $\hat V$ is also
full. We can for example introduce complicated sets of flat points in the
twisted cylinders, even as Cantor sets in the generating surfaces, but
these flat sets will be absorbed by a maximal full extension. As we will
show, any maximal full extension will satisfy the properties of $W$ in
the definition of \mbox{\ggm\!\!}, see \tref{mainfinite}. We expect that
the methods developed to prove this can be extended for distributions of
arbitrary rank.

\smallskip

The assumption of local finiteness in \tref{m2}
can be regarded as a mild regularity condition.
But we believe that even without
regularity conditions it should be possible to understand the gluing
between the twisted cylinders. We state:

\begin{conj}
If the set of nonflat points of a complete CN2 manifold with finite
volume admits a dense (not necessarily locally finite) extension, then
the complement of any maximal one is a disjoint union of compact totally
geodesic flat hypersurfaces, possibly accumulating (see Example~3 in
\sref{examples}).
\end{conj}

Certainly more difficult, we can ask what happens if we remove
all hypothesis on $V$. In particular, we do not know if the following
is true:

\begin{question*}
Does the set $V$ of nonflat points of a complete CN2 manifold with finite
volume admit a maximal (not necessarily dense or locally finite)
extension $\hat V$ such that $\partial \hat V$ is a union of flat totally
geodesic hypersurfaces, each of which has complete totally geodesic
boundary (if nonempty)? (See Example 2 in \sref{examples}).
\end{question*}

On the other hand, in the case of nonnegative or nonpositive curvature we
believe that no extra assumptions are needed:

\begin{conj}
Every compact CN2 manifold with nonnegative or nonpositive
scalar curvature and finite volume is a \ggm\!\!.
\end{conj}

In \cite{fz1} we classify all \ggms with nonnegative scalar curvature and
show that they are three dimensional up to a Euclidean factor. Moreover,
they are built as the union of at most two cylinders and, in particular,
are diffeomorphic to a lens space or a prism manifold.

\smallskip

Another interesting question is to what extent complete CN2 manifolds
with finite volume differ from \ggms from a
differentiable point of view:

\begin{question*} If $M^n$ admits a complete CN2 metric with finite
volume, does it also admit a \ggm metric?
\end{question*}

We caution that our definition of a graph manifold in dimension 3 is
more special than the usual topological one, where the pieces are
allowed to be nontrivial Seifert fibered circle bundles (\cite{w}). Ours
is similar, although more general, to the kind of graph manifolds one
studies in nonpositive curvature.

\vspace{5pt}

The paper is organized as follows. In \sref{examples} we provide some
examples in order to show that the two hypothesis in \tref{m2} are
necessary. A general semi-global version of the de Rham theorem is
provided in \sref{dr} and will be used in \sref{ptha} to prove \tref{m1}.
The proof of \tref{m2} is carried out in \sref{mainsec}.

\section{Examples.}\label{examples}

We now build some examples to help understand how \ggms are linked with
the CN2 property, and to what extent they differ. In particular, we
exhibit CN2 metrics on the 3-torus $T^3$ which are $C^\infty$
perturbations of the flat metric but that are not \ggm metrics.

\medskip

{\bf 1.} {\it The 3-torus as a nontrivial \ggm\!\!.}
Let $\su=[-1,1]^2$ with metric a $C^\infty$ perturbation of the flat
metric in a small open set $U\subset \su$ whose closure is contained in
the interior of $\su$. The cube $C=\su\times [-1,1]$ with its product
metric serves as a building block in all further examples, where the
second factor will give rise to the nullity foliations. Depending on the
example, we also adjust their sizes appropriately. Notice that the metric
necessarily has scalar curvature of both signs. We now glue two such
cubes along a common face in such a way that the nullity distributions
are orthogonal, see Figure 3. Identifying opposite faces of the resulting
larger cube defines a metric on $T^3$ with complete nullity foliations,
making it into a nontrivial graph manifold. Figure 3 shows the nonflat
points on the left, together with a full maximal extension and its two
(un)twisted cylinders on the right.
\begin{figure}[H]
\centering
\includegraphics[width=0.3\textwidth]{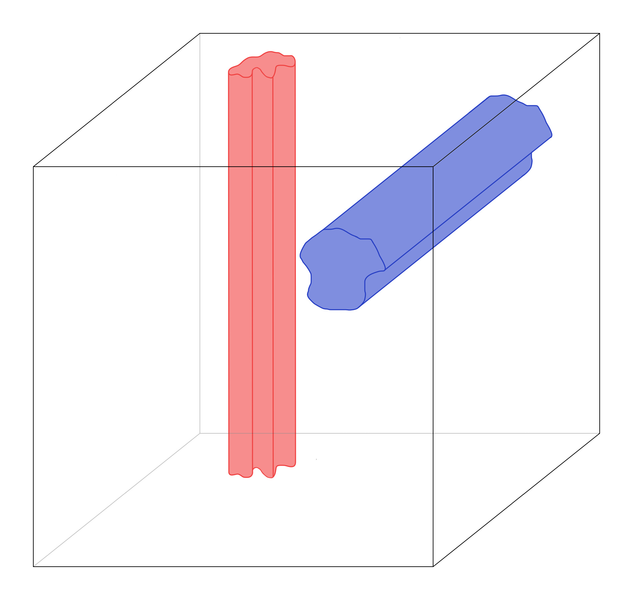}
\hspace{0.5cm}
\includegraphics[width=0.3\textwidth]{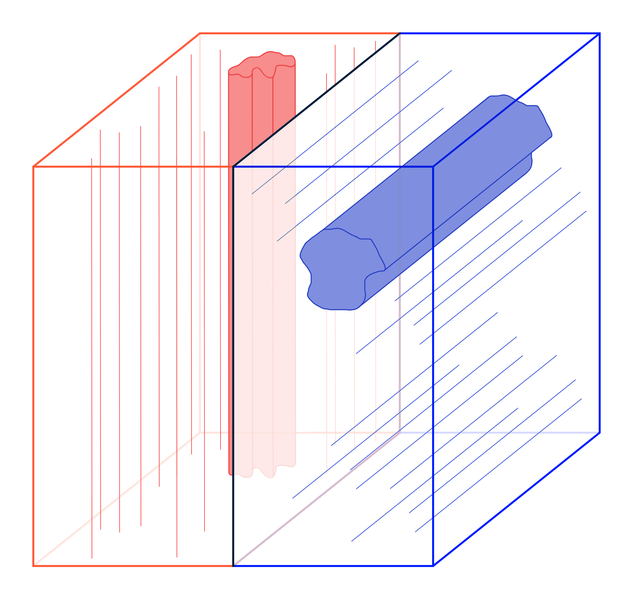}
\caption{\small A CN2 3-torus with its set of nonflat points,
and a full extension}
\end{figure}

\medskip

{\bf 2.} {\it A CN2 3-torus failing to be a \ggm\!\!: no maximal
dense extensions.}
Here we take three basic building blocks and glue them together as in
Figure 4. Adding two small flat cubes, we obtain a larger cube and
identifying opposite faces defines a CN2 metric on $T^3$. But this is not
a \ggm since the nullity distribution cannot be extended to a dense set
of $T^3$. Figure 4 shows the set of nonflat points on the left, and a
maximal extension of it on the right missing two octants of flat points.
We point out that this example also shows that a CN2 manifold does not
necessarily admit a~$T$ structure (see \cite{cg}), as graph manifolds do.
\begin{figure}[!ht]
\centering
\includegraphics[width=0.3\textwidth]{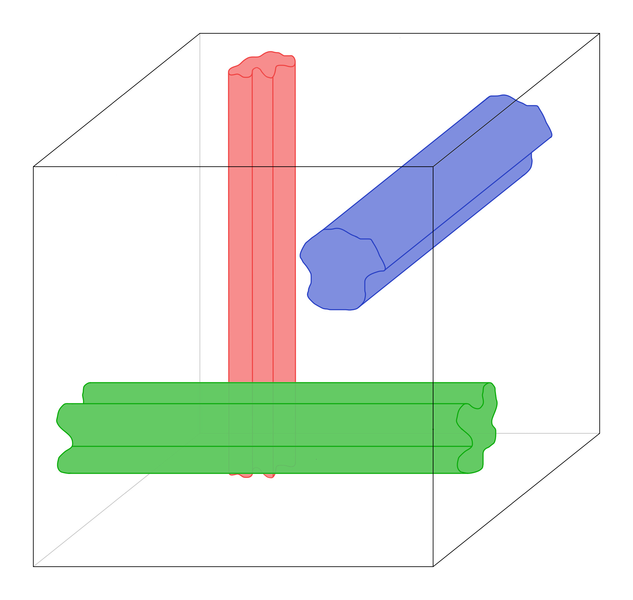}
\hspace{0.5cm}
\includegraphics[width=0.3\textwidth]{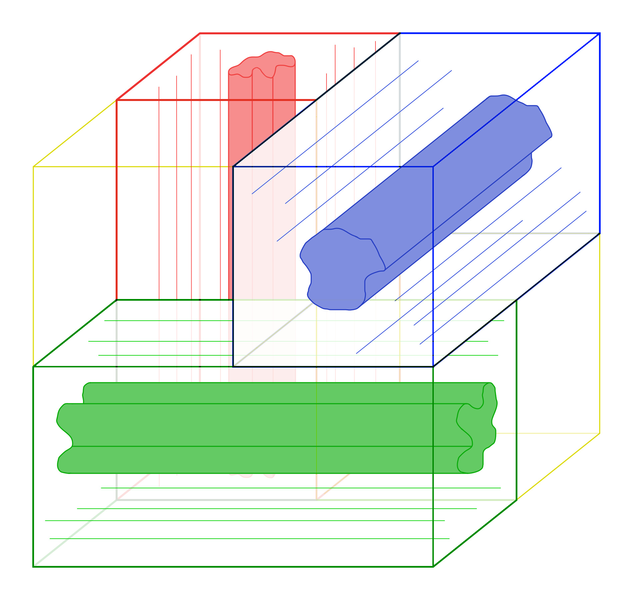}
\caption{\small A CN2 3-torus with its nonflat points
and a nondense maximal extension}
\end{figure}
\noindent

\medskip

{\bf 3.} {\it A CN2 3-torus failing to be a \ggm\!\!: no locally
finite extension.}
Take a sequence of building blocks $C_n=L_n\times[-1,1]$ with
$L_n=[-1/2^n,1/2^n]\times[-1,1]$. Glue one to the next along the squares
$\{\pm 1/2^n\}\times[-1,1]\times[-1,1]$ as in Example 1, with nullity
lines meeting orthogonally from one to the next, and accumulating at
a two torus $T^2=[-1,1]\times[-1,1]$. Now glue to this a copy of itself
along $T^2$. Identifying opposite sides in the resulting cube defines a
CN2 metric on $T^3$. The metric has two sequences of parallel totally
geodesic flat 2-tori, approaching $T^2$ from both sides. It is not a \ggm
since the number of connected components near $T^2$ of any extension
of~$V$ is infinite, even though there exist dense maximal extensions of
$V$. Notice that $T^2$ is disjoint from the union of the closures of the
connected components of any extension of $V$.

\medskip

{\bf 4.} {\it Drunken cylinders}.
We can modify the previous examples to obtain a more complicated behavior
of the twisted cylinders, illustrating another crucial difficulty in
trying to prove Conjecture 1 in the introduction. For this, we start with
a flat building block $C_n=L_n\times[-1,1]$ and identify two opposite
faces to to obtain a flat metric on $[-1/2^n,1/2^n]\times T^2$. Let
$\gamma_n$ be a closed geodesic in the two torus $\{0\}\times T^2$ and
modify the metric in a small tubular neighborhood of $\gamma_n$ in
$(-1/4^n,1/4^n)\times T^2$. The boundary of the resulting manifold
consists of two flat square tori and we can hence glue one to the next as
in the previous example. Gluing a mirror copy of the result and
identifying the remaining two faces, we obtain a CN2 metric on~$T^3$. We
can now choose the infinite sequence of cylinders comprising the example
such that the slopes of $\gamma_n$ converge. Notice that in the previous
examples the totally geodesic 2-tori separating components of a full
extension of the set of nonflat points have the property that they are
foliated by two different limits of nullity lines, whereas here the
torus~$T^2$ in the middle only has one family of limit nullity lines. The
existence of two families is crucial in proving convexity properties of
the boundary of a maximal full extension of the set of nonflat points;
see \sref{mainsec}.

\medskip

We finish this section with some examples that illustrate some of the
complexities of twisted cylinders.

\medskip

$a)$
Let $\Sigma$ be a compact surface with boundary with universal cover
$\su$, on which $G=\pi_1(\Sigma)$ acts freely. Now choose a homomorphism
$\alpha\colon G\to\R$ and an action of $G$ on $\su\times\R$ given by
$(x,t)\to (gx,t+\alpha(g))$. In the cylinder $C=(\su\times\R)/G$ the
integral leaves of $\Gamma^\perp$ are dense in $C$ as long as the values
of $\Im\alpha\subset\R$ are not all rationally dependent. Notice though
that $C$ is diffeomorphic to $\Sigma\times S^1$ by changing the action
continuously using the homomorphism $\e\alpha$ and letting $\e\to 0$.

\smallskip

$b)$
The boundary geodesics of the leaves of $\Gamma^\perp$ may not be closed.
As an example, start with a complete flat strip $[-1,1]\times\R$, remove
infinitely many $\e$-discs centered at $(0,n)$, $n\in \Z$, and change the
metric around them to make their boundaries totally geodesic in such a
way that the metric remains invariant under $\Z=\la h\ra$ for
$h(x,t)=(x,t+1)$. If $\su$ is this surface and $R$ is a rotation of $S^1$
of angle $r\pi$ for some irrational number $r$, then $G=\Z=\la(h,R)\ra$
acts freely and properly discontinuously on $\su\times S^1$. Although
$\su/G_1$ has three closed geodesics as boundary, the boundary of the
leaves of $\Gamma^\perp$ in $C=(\su\times S^1)/G$ has one closed
geodesic and also two complete open geodesics, each of which is dense
in the corresponding boundary component of $C$. 

\smallskip

$c)$
In general, a cylinder of finite volume does not
necessarily have compact boundary. Indeed, consider the metric
$ds^2=dr^2+e^{-\frac{1}{1-r^2}-t^2}dt^2$ on $\su=[-1,1]\times \R$.
Then $\su\times T^{n-2}$ has finite volume with two complete noncompact
flat totally geodesic boundary components. Nevertheless, we will see that
when we glue two twisted cylinders in a nontrivial way,
 their common boundary is 
compact.

\section{A semi-global de Rham theorem}\label{dr} 

The existence of a parallel smooth distribution $\Gamma$ on a complete
manifold $M^n$ implies that the universal cover is an isometric product
of a complete leaf of $\Gamma$ with a leaf of $\Gamma^\perp$ by the
global deRham theorem. In this section we prove a semi-global version of
this fact to be used later on, and will concentrate on flat foliations
for simplicity. Although this is a special case of Theorem 1 in
\cite{pr}, our proof is simpler and more direct, so we add it here for
completeness.

\begin{prop}\label{derham}
Let $W$ be an open connected set of a complete Riemannian manifold~$M^n$,
and assume that $\Gamma$ is a rank $k$ parallel distribution on $W$ whose
leaves are flat and complete. If $\D$ is a maximal leaf of
$\Gamma^\perp$, then the normal exponential map
$\exp^\perp:T^\perp\D\to W$ is an isometric covering, where $T^\perp\D$
is equipped with the induced connection metric. In particular, $W$ is
isometric to the twisted cylinder $(\tilde\D\times \R^k)/G$, where
$\tilde \D$ is the universal cover of $\D$, and $G$ acts isometrically in
the product metric.
\end{prop}
\proof
Since $\Gamma$ is parallel, its orthogonal complement $\Gamma^\perp$
is also parallel and hence integrable with totally geodesic leaves.
Due to the local isometric product structure of $W$ given by the local
deRham Theorem, the normal bundle $T^\perp\D$ of $\D$ is flat with respect
to the normal connection of $\D$, which in our case is simply the
restriction of the connection on~$M^n$ since $\D$ is totally geodesic.
Notice though that $T^\perp\D$ does not have to be trivial.
The normal connection defines, in the usual fashion, a connection metric
on the total space. By completeness of the leaves of $\Gamma$, the normal
exponential map $\exp^\perp\colon T^\perp\D\to W$ is well defined on the
whole normal bundle.

We first show that $\exp^\perp$ is a local isometry. Indeed, if
$\alpha(s)=(c(s),\xi(s))$ is a curve in $T^\perp\D$ with $\xi$ parallel
along $c$, then $\alpha'(0)$ is a horizontal vector in the connection
metric, identified with $c'(0)$ under the usual identification
$H_{\xi(0)}\simeq T_{c(0)}\D$. Then
$d(\exp^\perp)_{\alpha(0)}(\alpha'(0))=J(1)$ where $J(t)$ is the Jacobi
field along the geodesic $\gamma(t)=\exp^\perp(t\xi(0))$ with initial
conditions $J(0)=c'(0)$ and $J'(0)=0$ since $\xi$ is parallel along~$c$.
Since $\gamma'\in\ker R$, the Jacobi operator
$R(\,\cdot\,,\gamma')\gamma'$ vanishes and thus $J(1)$ is the parallel
translate of $J(0)\in T_{c(0)}L$, which in turn implies that $\exp^\perp$
is an isometry on the horizontal space. On the vertical space it is an
isometry since it agrees with the exponential map of the fiber which is a
complete flat totally geodesic submanifold of $W$. The image of
horizontal and vertical space under $\exp^\perp$ are also orthogonal,
since the first is the parallel translate of $TL$ and the second the
parallel translate of $T^\perp L$ along $\gamma$. Hence $\exp^\perp$ is a
local isometry. Notice that this also implies that if $\xi$ is a
(possibly only locally defined) parallel section of $T^\perp L$, then
$\{\exp_p(t\xi(p))\mid p\in L\}\subset W$ are integral manifolds of
$\Gamma^\perp$ for all $t\in\R$.

Let $q\in W$ and denote by $L_q$ the maximal leaf of $\Gamma^\perp$
containing $q$. Since the normal exponential map of $L_q$ is also a local
isometry, there exists an $\e=\e(q)>0$ such that $B_\e(q)\subset L_q$ is
a normal ball and the set $\hat V_q:=\{\xi\in T^\perp B_\e(q):
\|\xi\|<\e\}\subset T^\perp L_q$ is isometric to the Riemannian product
$B_\e(q)\times B_\e\subset L_q\times\R^k\cong L_q\times T^\perp_qL_q$,
and $\exp^\perp:\hat V_q\to V_q:=\exp^\perp(\hat V_q)$ is an isometry.
In particular, $q\in V_q\subset W$. We identify $B_\e(q)\times B_\e$,
$\hat V_q$ and~$V_q$ via the normal exponential map of $L_q$.
Accordingly,
we denote the local leaf of~$\Gamma^\perp$ through $x=(p,v)\in V_q$
by $L_{x,q}:=B_\e(q)\times\{v\}\subset L_x\cap V_q$.
Moreover, for each $y\in V_q$ and $v\in T^\perp_yL_{y,q}$ there is a
parallel vector field $\xi$ in $V_q$ with $\xi(y)=v$, and an isometric
flow $\phi^\xi_t(x)=\exp_x(t\xi(x))$ for $x\in V_q$. Notice that this
flow is defined for all $t\in\R$, and that the images of the leaves
$L_{x,q}$ are again leaves of $\Gamma^\perp$ for all $t\in\R$,
$x\in V_q$.

We now claim that $\exp^\perp$ is surjective onto $W$. Take a point
$q\in W$ in the closure of the open set $U=\exp^\perp(T^\perp\D)$ in $W$,
and choose $y\in V_q\cap U$. Since the leaf of $\Gamma$ containing~$y$
is also contained in the image of $\exp^\perp$, we can assume that
$y\in L_{q,q}$. Then $y=\gamma(1)$, where $\gamma(t)=\exp^\perp(t\xi)$,
for $x\in L$ and $\xi\in T_x^\perp L$. If $\eta$ is the parallel vector
field in $V_q$ with $\eta(y)=\gamma'(1)\in T^\perp_yL_{q,q}$, then
$\phi^\eta_{-1} (L_{q,q})\subset L$ by maximality of $L$ and hence
$q\in U$. Hence, $U$ is closed in $W$, so $U=W$.

In order to finish the proof that $\exp^\perp$ is a covering map, we show
that it has the curve lifting property. Let $\alpha:[a,b]\to W$ be a
smooth curve, and assume there exists a lift of $\alpha|_{[a,r)}$ to a
curve $\tilde \alpha :[a,r)\to T^\perp L$. As usual, to extend
$\tilde \alpha$ past $r$ we only need to show that $\lim_{t\to
r}\tilde\alpha(t)$ exists. Write
$\tilde\alpha(t)=(c(t),\xi(t))\in T^\perp L$. Since $\exp^\perp$ is a
local isometry, $|\tilde \alpha'|=|\alpha'|$ and hence by the local
product structure $c$ and $\xi$ have bounded length. Thus,
$\lim_{t\to r}c(t)=x_\infty\in M$ and $\lim_{t\to r}\xi(t)=\xi_\infty$
exist by completeness of $M^n$ and $\R^k$. We only need to show that
$x_\infty\in L$ since then $\exp^\perp(x_\infty,\xi_\infty)=\alpha(r)$.

Consider $\delta>0$ such that $\alpha((r-\delta,r])\subset V_{\alpha(r)}$.
For $t\in (r-\delta,r)$ let $\eta_t$ be the parallel vector field in
$V_{\alpha(r)}$ with
$\eta_t(\alpha(t))=\gamma_t'(1)\in T^\perp L_{\alpha(t)}$ where
$\gamma_t(s)= \exp^\perp(s\xi(t))$. We then have
$\lim_{t\to r}\eta_t=\eta_\infty\in T^\perp L_{\alpha(r)}$ with
$\exp_{\alpha(r)}(-\eta_\infty)=x_\infty\in W$ and hence
$\phi^{\eta_\infty}_{-1}(L_{\alpha(r),\alpha(r)} )\subset L_{x_\infty}$.
Since we also have that
$\phi^{\eta_t}_{-1}(L_{\alpha(t),\alpha(r)})\subset L$ for $t<r$,
it follows that $\lim_{t\to r}T_{c(t)}L=T_{x_\infty}L_{x_\infty}$ and
thus $L\cup L_{x_\infty}$ is an integral leaf of $\Gamma^\perp$. By the
maximality of $L$ we conclude that $L_{x_\infty}\subset L$ and therefore
$x_\infty\in L$, as we wished.

Finally, if $\pi\colon \tilde L\to L$ is the universal cover, then
$\pi^*(T^\perp L)\to T^\perp L$ is also a cover and since $\pi^*(T^\perp
L)$ is again a flat vector bundle over a simply connected base, it is
isometric to $\tilde L\times\R^k$. This proves the last claim.
\qed

\section{The structure of the set of nonflat points}\label{ptha} 

This section is devoted to the proof of the following stronger version of
\tref{m1}.

\begin{theorem}\label{product}
Let $M^n$ be a complete CN2 manifold, and $V$ a connected component of
the set of nonflat points of $M^n$. If $V$ has finite volume, then its
universal cover is isometric to $\gencylrec$, where
$\su$ is a simply connected surface whose Gauss curvature is
nowhere zero and vanishes at its boundary.
\end{theorem}
\begin{proof}
First, recall that, since $\Gamma=\ker R$ is a totally geodesic
distribution, we have its {\it splitting tensor} $C:\Gamma\to
End(\Gamma^\perp)$ defined as
$$
C_TX=-(\nabla_XT)_{\Gamma^\perp},
$$
where $\nabla$ is the Levi-Civita connection of $M^n$, and a distribution
as a subscript means to take the corresponding orthogonal projection.
Clearly, $\Gamma^\perp$ is totally geodesic if and only if $C\equiv0$,
that is equivalent to the parallelism of $\Gamma$. Since the set of
nonflat points in a CN2 manifold agrees with the points of minimal
nullity, $V$ is saturated by the flat complete leaves of $\Gamma$.
Therefore, by \pref{derham}, all we need to show is that $C$ vanishes.

\smallskip

Let
$U,S\in\Gamma$ and $X\in\Gamma^\perp$. Since $\Gamma$ is totally
geodesic,
\begin{align*}
C_{\nabla_US}X&=-(\nabla_X\nabla_US)_{\Gamma^\perp}=
-(\nabla_U\nabla_XS)_{\Gamma^\perp}-(\nabla_{[X,U]}S)_{\Gamma^\perp}\\
&=(\nabla_U(C_SX))_{\Gamma^\perp}+C_S([X,U]_{\Gamma^\perp})
=(\nabla_UC_S)X+C_S(\nabla_UX)-C_S([U,X]_{\Gamma^\perp})\\
&=(\nabla_UC_S)X+C_S(\nabla_XU)=(\nabla_UC_S)X-C_SC_UX,
\end{align*}
or
\begin{equation}\label{superric}
\nabla_UC_S=C_{\nabla_US}+C_SC_U,\ \ \forall\ U,S\in\Gamma.
\end{equation}
We now consider the so called {\it nullity geodesics}, i.e. complete
geodesics $\gamma$ with $\gamma'(0)\in\Gamma$, which are hence
contained in a leaf of $\Gamma$. Along such a geodesic $\gamma$,
by \eqref{superric} the splitting tensor $C_{\gamma'}$ satisfies the
Riccati type differential equation
$\nabla_{\gamma'}C_{\gamma'}=C_{\gamma'}^2$
over the entire real line. That is, with respect to a parallel basis,
\begin{equation}\label{sol}
C_{\gamma'}'=C_{\gamma'}^2, \text{\ whose solutions are }
C(t)=C_0(I-tC_0)^{-1},
\text{\ for }\ C_0:=C_{\gamma'(0)}.
\end{equation}
Therefore, along each nullity geodesic $\gamma$ in $V$, all real
eigenvalues of $C_{\gamma'}$ vanish. Since $\Gamma^\perp$ is
2-dimensional, for every $S\in\Gamma$ either all eigenvalues of $C_S$ are
complex and nonzero, or all eigenvalues are $0$, i.e. $C_S$ is nilpotent.
In particular, $C_S$ vanishes if it is self adjoint.

Let $W=\{p\in V: C\neq 0 \text{ at } p\}$, i.e. on $V\setminus W$ all
splitting tensors vanish. Since the space of self-adjoint endomorphisms
of $\Gamma^\perp$ is pointwise 3-dimensional and intersects
$\im C\subset End(\Gamma^\perp)$ only at 0, it follows that $\dim\im C=1$
in $W$, and hence $\ker C$ is a smooth codimension 1 distribution of
$\Gamma$ along $W$. Accordingly, write
$$
\Gamma=\ker C\oplus^\perp \spa\{T\},
$$
for a unit vector field $T\in\Gamma$, which is well defined, up to sign,
on $W$. By going to a two-fold cover of $W$ if necessary, we can
assume that $T$ can be chosen globally on $W$.

Observe that if $U,S$ are two sections of $\ker C$, then \eqref{superric}
implies that $\nabla_US\in \ker C$, i.e. $\ker C$ is totally geodesic,
and $\nabla_TU=0$ as well. Since $\Gamma$ is totally geodesic it follows
that $\nabla_TT=0$, that is, the integral curves $\gamma$ of $T$ are
nullity geodesics. Therefore, from now on let for convenience $C=C_T$,
$C(t)=C_{\gamma'(t)}$, and denote by $'$ the derivative in direction of
$T$. In particular,
\begin{equation}\label{div}
\text {div } T=\tr\nabla T=-\tr C.
\end{equation}
By \eqref{sol} we have
\begin{equation}\label{tr}
\tr C(t)=\frac{\tr C_0-2t\det C_0}{1-t\tr C_0+t^2\det C_0},\ \
\text{ and }\ \
\det C(t)=\frac{\det C_0}{1-t\tr C_0+t^2\det C_0}.
\end{equation}
Take $B\subset W$ a small compact neighborhood. Since either $\det C>0$
or $\det C=\tr C=0$ on $W$, by \eqref{tr} there is $t_0\in\R$ such that
$\tr C(t)(q)\leq 0$ for every $q\in B$ and every $t\geq t_0$. In
addition, defining $B_t:=\phi_{t+t_0}(B)$ and $v(t):=\vol B_t$
we have that
$$
v'(t)=\int_B\frac{d}{dt}\phi_t^*(dvol)=\int_B\text{div }T
=-\int_B\tr C\geq0,\ \ \ \forall\ t\geq 0.
$$
So, the sequence of compact neighborhoods
$\{B_{n}, n\in \N\}$ has nondecreasing volume in the set $V$ of
finite volume, and thus there is a strictly increasing sequence
$\{n_k:k\geq 0\}$ such that $B_{n_k}\cap B_{n_0}\neq\emptyset$
for all $k\geq 1$. We will refer to this property as
{\it weak recurrence}. In particular, there exists a sequence
$p_k:=\phi_{t_0+n_k}(q_k)\in B_{n_k}\cap B_{n_0}$, with $q_k\in B$,
which has an accumulation point $p\in B_{n_0}\subset W$.

Consider the open subset $W'\subset W$ on which $C$ has nonzero complex
eigenvalues and notice that, by \eqref{sol}, $W'$ is invariant under the
flow $\phi_t$ of $T$. Using the above recurrence and sequence of points
$p_k\to p$, \eqref{tr} implies that
$$
\det C_{T(p)}=\lim_{k\to+\infty}\det C_{T(p_k)}=\lim_{k\to+\infty}
\frac{\det C_{T(q_k)}}{1-(t_0-n_k)\tr C_{T(q_k)}+(t_0-n_k)^2
\det C_{T(q_k)}}=0,
$$
since $n_k\to+\infty$ and $q_k$ lies in the compact set $B$.
But this contradicts the fact that $p\in B_{n_0}\subset W'$, where
$\det C>0$. Thus $C$ vanishes on $W'$, which is a contradiction and
shows that $C$ is nilpotent on $W$.

We thus have a well defined 1-dimensional distribution on $W$ spanned by
the kernel of $C=C_T$, which is parallel along nullity lines by
\eqref{sol}. Replacing $W$, if necessary, by the two-fold cover where
this distribution has a section, and by a further cover to make $W$
orientable, we can assume that there exists an orthonormal basis
$e_1,e_2$ of $\Gamma^\perp$, defined on all of $W$, and parallel along
nullity lines with
\begin{equation*}
C(e_1)=0, \ \ C(e_2) = a e_1.
\end{equation*}
Hence
$$
\nabla_Te_1=\nabla_Te_2=\nabla_TT=0,\quad\nabla_{e_1}T=0,
\quad\nabla_{e_2}T=-ae_1,
$$
$$
\nabla_{e_1}e_1=\alpha e_2,\quad \nabla_{e_2}e_2=\beta e_1,
\quad \nabla_{e_1}e_2=-\alpha e_1,\quad \nabla_{e_2}e_1=aT-\beta e_2,
$$
for some smooth functions $\alpha,\beta$ on $W$. A calculation shows that
\begin{align*}
R(e_2,e_1)e_1&=(e_1(\beta)+e_2(\alpha)-\alpha^2-\beta^2)e_2 + (a\beta-e_1(a))T,\\
R(e_1,e_2)e_2&=(e_1(\beta)+e_2(\alpha)-\alpha^2-\beta^2)e_1 + \alpha aT,
\end{align*}
and hence
\begin{equation}\label{scal}
\alpha=0,\ \ \Scal_M=e_1(\beta)-\beta^2,\ \ \text{and}\ \ e_1(a)=a\beta,
\end{equation}
where $\Scal_M$ stands for the scalar curvature of $M^n$.
The differential equation \eqref{sol} implies that $a'=0$, i.e. $a$ is
constant along nullity lines. Thus
$$
e_2(a)'=e_2(a')+[T,e_2](a)=(\nabla_Te_2-\nabla_{e_2}T)(a)=ae_1(a),
$$
and $e_1(a)'=e_1(a')+[T,e_1](a)=0$. So, $e_2(a)=ae_1(a)t+d$ for some
smooth function $d$ independent of $t$. By the weak recurrence property,
$e_1(a)(p)=0$ for $p\in B_{n_0}$ as above, and hence
$e_1(a)(p')=0$ as well, for $p'=\phi_{-n_0-t_0}(p)\in B$. Since $B$
is arbitrary small, we have that $e_1(a)\equiv 0$ on $W$.
But then
\eqref{scal} implies that $\beta=0$ and thus $\Scal_M=0$, contradicting
that along $ V$ we have no flat points. Altogether, $C$ vanishes
everywhere on $V$ and hence $\Gamma$ is parallel.
\end{proof}
\vspace{1.5ex}

The proof becomes particularly simple for $n=3$ since then $\Gamma$
is one dimensional and only the differential equation $C'=C^2$ along the
unique nullity geodesics is needed. The local product structure in this
case was proved earlier in \cite{sw} with more delicate techniques.

\bigskip

Let us finish this section with some observations about the
geometric structure when the manifold is complete but the finite volume
hypothesis is removed.

\begin{rem*}
$a)$ If $C_T$ is nilpotent, then $\Scal$ is constant along nullity
leaves. If this constant is positive, and $M^n$ is complete, then the
universal cover is isometric to $\su\times \R^{n-2}$. This follows since
by \eqref{scal}~$a^{-1}$ satisfies the Jacobi equation
$(a^{-1})''+a^{-1}\Scal=0$ along integral curves of $e_1$, which cannot
be satisfied for all $t\in\R$. For $n=3$ this splitting was also proved
in \cite{am}.

$b)$ If $C_T$ does not vanish and has complex eigenvalues, then
$\Scal(t)=\Scal(0)/(1-t\tr C_0+t^2\det C_0) $ along nullity lines since
one easily sees that $\Scal'=(\tr C_T) \Scal$. As was shown in \cite{sz},
if $M^n$ is complete and $C_T$ has no real eigenvalues, then the
universal cover of $M^n$ splits off a Euclidean space of dimension
$(n-3)$, i.e. all locally irreducible examples are 3-dimensional.
\end{rem*}

\section{CN2 manifolds as \ggms}\label{mainsec} 

The purpose of this section is to prove \tref{m2} in the introduction.
As we will see, the proof is quite delicate and technical due to the
lack of any {\it a priori} regularity of the boundary of a maximal full
extension. The strategy is to consider a maximal extension of the set
of nonflat points which, by \tref{m1}, is a union of twisted cylinders,
and then to analyze their geometric properties at contact points.

\hspace{0.4cm}

We begin with some definitions.

\begin{definition}\label{lf}
{\rm We say that an open subset $W$ of a topological space $M$ is
{\it locally finite} if, for every $p\in \partial W$, there exists an
integer $m$ such that, for every neighborhood $U\subset M$ of $p$, there
exist a neighborhood $U'\subset U$ of $p$ such that $U'\cap W$
has at most $m$ connected components. We denote by $m(p)$ the minimum of
such integers $m$.}
\end{definition}

In this situation, for every $p\in\pw$, there exists a neighborhood
$U_p\subset M$ of $p$ such that $U_p\cap W$ has precisely $m(p)$
connected components. Notice that each component contains~$p$ in its
boundary since by finiteness all other components have distance to $p$
bounded away from $0$. We call these components $W_i$ the
{\it \lccs of $W$ at $p$}. Notice also that $U_p$ can be chosen
arbitrarily small. In fact, given any neighborhood $U$ of $p$ with
$U\subset U_p$ we can construct a new neighborhood $U_p(U)$ of $p$ as
follows. Let $X$ be the union of the connected components of $U\cap W$
that contain $p$ in their boundary. Then $U_p(U)=(\overline X)^\circ$ is
a neighborhood of $p$ (by local finiteness) and $U_p(U)\cap W$ has $m(p)$
connected components. Observe also that $U_p(U)\cap W_i$ are the
connected components of $U_p(U)\cap W$ and all contain~$p$ in their
boundary. Throughout this section $U_p$ will always denote such a
neighborhood of~$p$ and $W_i$ the connected components of $U_p\cap W$.

In particular, by taking any $\delta>0$ such that
$B_\delta(p)\subset U_p(B_\e(p))$ we get:

\begin{lem}\label{e}
If $M^n$ is a Riemannian manifold, $W\subset M^n$ is locally finite and
$p\in\pw$, then for every ball $B_\e(p)\subset U_p$ there is
$0<\delta=\delta(\e,p)<\e$ such that $W_i\cap B_\delta(p)$ is
arc-connected in $W_i\cap B_\e(p)$, for all $1\leq i\leq m$.
\end{lem}

\smallskip

Let $M^n$ be a complete CN2 manifold whose nullity distribution is
parallel along the set of nonflat points $V$ of $M^n$, as is the case
when $M^n$ has finite volume by \tref{m1}. Suppose $V$ has a full
extension $W\subset M^n$, that is, $W\supset V$ is open, dense, locally
finite, and $W$ possesses a smooth parallel distribution $\Gamma$ of rank
$n-2$ whose leaves are flat and complete. Since any extension of a full
extension is also full, we will assume in addition that $W$ is maximal.
Clearly, along $V\subset W$, $\Gamma$ coincides with the nullity of
$M^n$. Observe that maximal extensions always exist by definition, but
they are not necessarily unique, and may fail to be dense or locally
finite as shown in Examples 3 and 4 in \sref{examples}. We call the
leaves of $\Gamma$ nullity leaves and, for simplicity, use $\Gamma(p)$
both for the distribution at~$p$ and for the leaf of $\Gamma$ through
$p$.

\smallskip

Observe that, for each sequence $\{p_n\}$ in $W$ approaching a point
$p\in\partial W$, $\Gamma(p_n)$ accumulates at an $(n-2)$-dimensional
subspace of $T_pM$ whose image under the exponential map gives a complete
totally geodesic submanifold of~$M^n$, by completeness of the leaves
of~$\Gamma$. We still denote the set of all these limit submanifolds by
$\Gamma(p)$, and call each of them a {\it boundary nullity leaf} at $p$,
or BNL for short. In addition, given $U\subset W$ with $p\in\partial U$,
we denote by $\Gamma_U(p)\subset\Gamma(p)$ the BNL's at $p$ that arise as
limits of nullity leaves in $U$. In particular, if $W_1,\dots,W_m$ are
the \lccs of $W$ at $p$, we have
$$
\Gamma(p)=\Gamma_{W_1}(p)\cup\cdots \cup \Gamma_{W_m}(p).
$$

We start with the following observation.

\begin{lem}\label{4}
Let $p\in \pw$ such that $\Gamma(p)$ has only one BNL $\mu$.
Then, $\Gamma(q)=\{\mu\}$ for all $q\in\mu$.
\end{lem}
\proof
By definition, there is a unique BNL $\mu\in\Gamma(p)$. The hypersurface
$B'_\e(p):=\exp(\mu^\perp\cap B_\e(0_p))$ is then transversal to $\Gamma$
in $W'_\e(p)=W\cap B'_\e(p)$, which is thus dense in $B'_\e(p)$.

Take $q\in\mu$, and write it as $q=\gamma_v(T)$ for some $v\in T_p\mu$,
$T\in\R$, $\|v\|=1$. By continuity of the geodesic flow at $v$, given
$\delta>0$, there is $\e>0$ such that all the nullity leaves of~$W$
through $W'_\e(p)$ are $C^0$ $\delta$-close to $\mu$ inside a compact
ball of radius, say, $2T$, centered at~$p$. In particular, these nullity
lines of $W'_\e(p)$ stay close to $\mu$ at $q$ and form an open dense
subset. This implies that there cannot be two different BNL's at $q$.
Indeed, a second $\mu'\in\Gamma(p)$ is a limit of leaves of $\Gamma$ of a
local connected component $W'$ at $q$. Thus all leaves in $W'$ are close
to $\mu'$ which implies that leaves of $\Gamma$ on $W$, where it is an
actual foliation, would intersect near $q$.
\qed
\vspace{1.5ex}

For the next two lemmas we need a relationship between curvature bounds
and local parallel transport for Riemannian vector bundles over surfaces.

\begin{lem}\label{integral}
Let $E^k$ be a Riemannian vector bundle with a compatible connection
$\nabla$ over a surface $S$, and let $D\subset S$ be a region
diffeomorphic to a closed 2-disk with piecewise smooth boundary~$\alpha$.
If the curvature tensor of $E^k$ is bounded by $\delta>0$ along $D$, then
the angle between any vector $\xi\in E_{\alpha(0)}$ and its parallel
transport along $\alpha$ is bounded by $(k-1)\,\delta$Area$(D)$.
\end{lem}
\proof
Let us consider polar coordinates on $D$ through a diffeomorphism with
a 2-disk. We can assume that $\xi$ is a unit vector and we complete
$\xi=\xi_1,\dots,\xi_k$ to an orthonormal basis of $E_{\alpha(0)}$. By
radially parallel transporting them first to $p$, and then radially to
all of $D$, we get an orthonormal basis, which we again denote by
$\xi_1,\dots,\xi_k$, defined on $D$. If we consider the connection
1-forms $w_{ij}(X)=\la\nabla_X\xi_i,\xi_j\ra$ on~$D$, then one easily
sees that $dw_{ij}=\la R_\nabla(\cdot,\cdot)\xi_i,\xi_j\ra$ since
$\dim D=2$ and $w_{ij}(Y)=0$ for the radial direction $Y$.

Let $\xi(t)=\sum_ia_i(t)\xi_i(\alpha(t))$ be the parallel transport of
$\xi$ along $\alpha$ between 0 and $t$. Then, since
$\nabla_{\alpha'}\xi=0$, we have
$a_1'=\la \xi_1,\xi\ra'= \la\nabla_{\alpha'}\xi_1,\xi\ra
=\la\nabla_{\alpha'}\xi_1,\sum_{i=1}^k a_i\xi_i\ra
=\sum_{i=2}^k a_i w_{1i}(\alpha')$.
Therefore, since $|a_i|\leq 1$ we obtain
$$
0\leq 1-\la\xi(0),\xi(1)\ra = - \int_0^1 \! a_1'=
-\sum_{i=2}^k\int_0^1\! a_i\,w_{1i}(\alpha')\leq\sum_{i=2}^k\int_\alpha|w_{1i}|.
$$
For each $i$, choose a partition of $\alpha$ into countable many
segments $\alpha_1,\alpha_2,\dots$ with $w_{1i}|_{\alpha_{2j}}\geq 0$
and $w_{1i}|_{\alpha_{2j+1}}\leq 0$. Then, $\alpha_j$ together with the
radial curves (along which $w_{1i}=0$) encloses a triangular region $T_j$
where we apply Stokes Theorem to get
$$
\int_\alpha|w_{1i}| = \sum_j \int_{T_{2j}}dw_{1i}-
\sum_j \int_{T_{2j+1}}dw_{1i} \leq
\int_D|\la R_\nabla(\cdot,\cdot)\xi_1,\xi_i\ra|\leq \delta Area(D).
$$
\vspace{-32pt\qed}
\vspace{5ex}

\smallskip

In the following lemmas we study the behavior of the nullity leaves and
BNL's in \mbox{$U_p\subset B_\e(p)$}. To do this, we will be able to
restrict the discussion to a single surface $S$ transversal to the
nullity leaves and BNL's near $p$ due to the following result;
see Figure 5.

\begin{lem}\label{surface}
For each $p\in\pw$ there exist a 2-plane $\tau\subset T_pM$,
a sufficiently small convex ball $B_\e(p)$, and a neighborhood
$U_p\subset B_\e(p)$ of $p$ such that the surface
$S:=\exp(\tau\cap B_\e(0_p))\subset B_\e(p)$ satisfies:
\begin{enumerate}[label=\alph*)\,]
\item $S$ intersects all nullity leaves and BNL's in $U_p$,
and does so transversely;
\item $W_i\cap S$ is connected and its closure is not contained in
$U_p\cap S$;
\item $U_p\cap S$ is diffeomorphic to an open disc.
\end{enumerate}
\end{lem}
\begin{proof} For a fixed $1\leq i \leq m(p)$, choose some BNL
$\mu\in\Gamma_{W_i}(p)$ and for some convex ball $ B_\e(p)$ consider the
surface $L=\exp(\mu^\perp\cap B_\e(0_p))$. Given $\delta>0$, we will fix
$\e>0$ such that
$(n-1)\max\{|\Scal_M(x)|: x\in B_{\e}(p)\}Area(L)<\delta$ and choose
$U_p\subset B_\e(p)$. Notice also that a bound on $\Scal_M$ gives a bound
on the full curvature tensor since $M^n$ is CN2.
Take a sequence $p_i\in W_i$ converging to $p$ such that
$\Gamma(p_i)\to\mu$. For $i$ large enough $\Gamma(p_i)$ is transversal to
$L$ and we can thus assume that $p_i\in W_i\cap L$. Furthermore, fix $i$
large enough such that for $q:=p_i$ the parallel translate of $\Gamma(q)$
along the minimal geodesic $\overline{qp}$ from $q$ to $p$ has angle less
than $\delta$ with $\mu$. Let $W_i'$ be the arc-connected component of
$W_i\cap L$ that contains~$q$. Thus for any $q'\in W_i'$, we can choose a
curve $\alpha\subset W_i'$ connecting $q$ to $q'$. If
$\beta=\overline{pq}$ and $ \beta'=\overline{q'p}$, we form the closed
curve $\varphi=\beta*\alpha*\beta'\subset L$. We can also choose $\alpha$
such that $\varphi$ is simple and hence bounds a disc $D\subset L$.
According to \lref{integral} the parallel transport of $\mu$ along
$\varphi$ forms an angle less than $\delta$ with $\mu$. In other words,
the parallel transport of $\Gamma(q')$ along $\overline{q'p}$ has angle
less than $2\delta$ with $\mu$ for any $q'\in W_i'$. We can thus choose
$\e'<\e$ sufficiently small, and $U_p\subset B_{\e'}(p)$ such that all
nullity leaves in $W_i'$ intersect~$L$ transversely at an angle bounded
away from $0$. We now claim that this implies that $W_i'=W_i$, i.e.
$W_i\cap L$ is connected. Otherwise, there exists an $x\in W_i$ with
$x\in \pw_i'$. Since $\Gamma(x)$ still intersects $L$ transversely, we
can choose a small product neighborhood $U\subset W_i$ as in the proof of
\tref{derham} such that $x\in U$ and all nullity leaves in $U$
intersect~$L$ in a unique point and transversely. But then any two points
in $U\cap L$ can be connected in $U$ and then projected along nullity
leaves to lie in $L$. Thus $U\cap L$ is also contained in~$W_i'$.

Since there are only finitely many connected components, and all
components of $U_p\cap W$ contain $p$ in their boundary, there exists a
common $\e'$ and BNL's $\mu_i\in\Gamma_{W_i}(p)$ satisfying the above
properties. We can now choose a 2-plane $\tau\subset T_pM$ transversal to
all $\mu_i$ and set $S:=\exp(\tau\cap B_{\e}(0_p))$. Repeating the above
argument for this surface $S$, we see that $\e$ can be chosen
sufficiently small such that all nullity leaves in $U_p=U_p( B_{\e}(p))$
intersect $S$ transversely and that for all its connected
components $S\cap W_i$ is connected as well. Since in addition we can
assume that the angle between the nullity lines and $S$ is bounded away
from $0$, all BNL's are transversal to $S$ as well. Notice also that now
the components of $W\cap U_p\cap S$ are precisely $W_i\cap S$.

So far $S$ and $U_p$ satisfy the properties in part $(a)$ and the first
part of $(b)$. For the second part of $(b)$, since $W$ is locally finite,
we simply choose $\e'<\e$ small enough such that the closure of the
connected components $W_i\cap S$ are not strictly contained in
$B_{\e'}(p)$. But then $U_p(B_{\e'}(p))$ is the desired neighborhood
(see Figure 5).

We now claim that such a neighborhood is also simply connected. Let
$\alpha$ be a closed curve in $U_p$ which bounds a disc $D\subset S$. If
$D$ contains a point in another component $W'$ of $W\cap S$, then $W'$ is
fully contained inside $D$ since it does not touch $\alpha$. Hence the
closure of $W'$ is contained in $U_p$, which contradicts $(b)$. Thus $D$
is contained in $U_p$, and we can find a null homotopy of $\alpha$ in
$D\subset U_p$.
\end{proof}
\vspace{0.5ex}

\begin{figure}[!ht]
\centering
\includegraphics[width=0.375\textwidth]{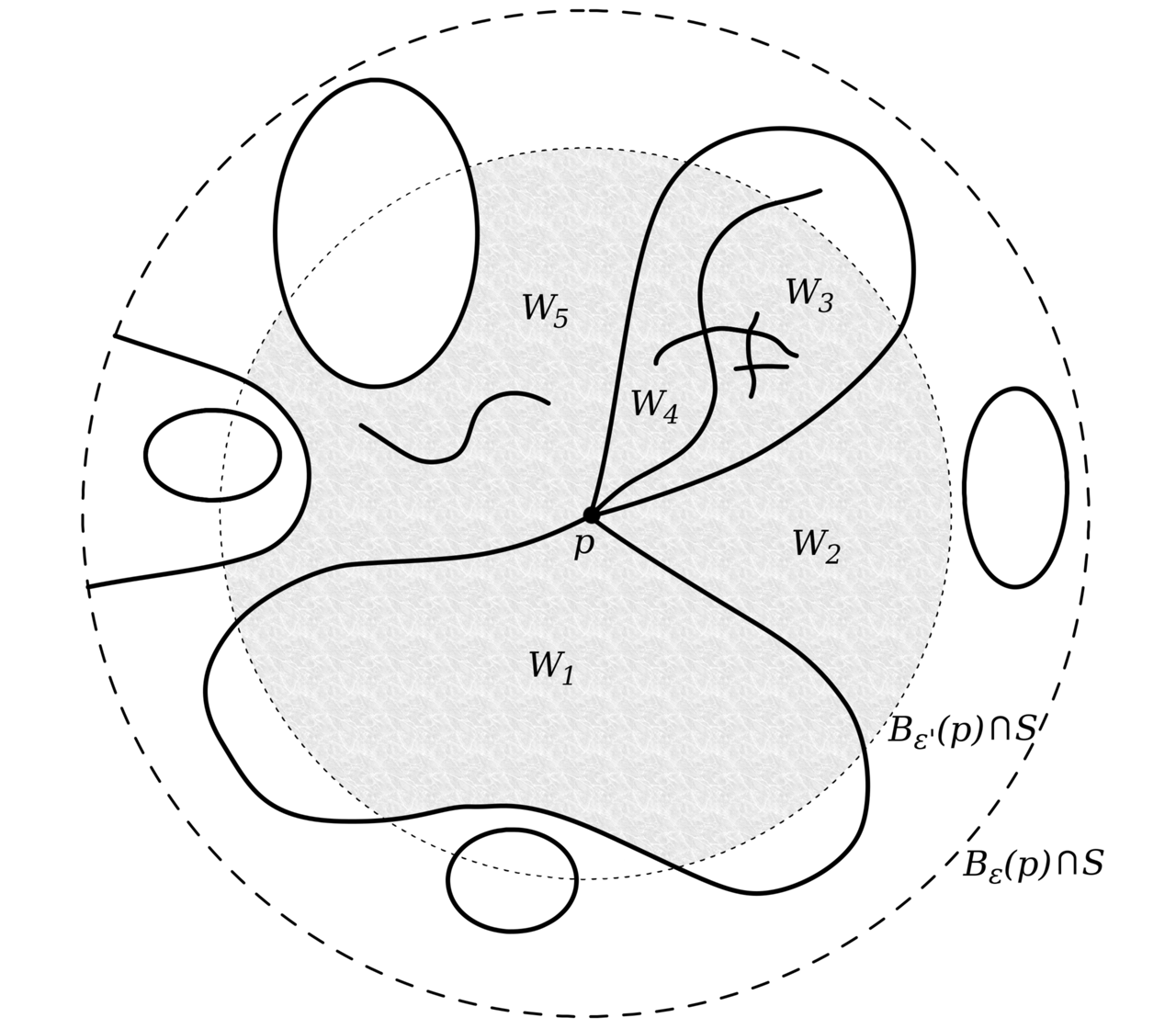}
\caption{\small A point $p\in\pw$ with $m(p)=5$,
the dark lines represent $\pw\cap S$,}
\centerline{\small while the shaded area corresponds to $U_p\cap S$}
\end{figure}

We will use $\e>0$, $S$, and $U_p$ as in \lref{surface} for the remainder
of this section. We point out though that the open sets $W_i\cap S$ can
have quite complicated boundary. In fact, $\pw_i$ may not even be a
Jordan curve and hence may not consist of a union of continuous arcs.
Furthermore, for $p\in\pw_i$ there may not even be a continuous curve in
$W_i$ with endpoint~$p$. We thus carefully avoid using any such
assumptions on properties of these boundaries.

It already follows from the proof of \lref{surface} that
all BNL's in $\Gamma_{W_i}(p)$ form a small angle. We will now show that
it is in fact unique.

\begin{lem}\label{only1}
If $W_i$ are the local connected components at $p$, then
$\Gamma_{W_i}(p)$ is a single BNL, for each $1\leq i \leq m$.
\end{lem}
\begin{proof}
Let $\mu_1,\mu_2\in \Gamma_{W_i}(p)$ be two BNL's at $p$ and two
sequences $p_{r,k}\in W_i\cap S$, $r=1,2$, converging to $p$ with
$\Gamma(p_{r,k})\to\mu_r$. For any $\e'>0$ choose $0<\delta(\e',p)<\e'$
as in \lref{e}. For $k$ large enough $p_{r,k}\in B_{\delta'}(p)$ and we
can choose a curve $\alpha_k\subset W_i\cap B_{\e'}(p)$ connecting
$p_{1,k}$ to $p_{2,k}$ and by \lref{surface} we can also assume that
$\alpha_k$ lies in $S$. Now define the loop
$\varphi_k=\beta_{2,k}*\alpha_k*\beta_{1,k}^{-1}\subset S \cap
B_{\e'}(p)$, where $\beta_{r,k}=\overline{p_{r,k}p}$. We can assume it is
a simple closed curve and hence encloses a 2-disk
$D\subset S\cap B_{\e'}(p)$. Therefore, \lref{integral} implies that the
angle between $\mu_1$ and its parallel transport along $\varphi_k$ is
bounded by $(n-1)Area(S)s(\e')$, where
$s({\e'}):=\max\{|\Scal_M(x)|: x\in B_{\e'}(p)\}$. On the other hand, the
parallel transport of $\Gamma(p_{r,k})$ along $\beta_{r,k}$ converges to
$\mu_r$ as $k\to \infty$ and the parallel transport of $\Gamma(p_{1,k})$
along $\alpha_k$ is equal to $\Gamma(p_{2,k})$. Hence the angle between
$\mu_2$ and the parallel transport of $\mu_1$ along $\varphi_k$ goes to
$0$ as $k\to\infty$. Finally, $s({\e'})\to 0$ as ${\e'}\to 0$ since
$\Scal_M(p)=0$ and we conclude that $\mu_1=\mu_2$ as $\e'\to 0$.
\end{proof}

\begin{lem}\label{constant}
For $q\in \pw_i\cap\pw_j\cap S$, both $\Gamma_{W_i}(q)$ and
$\Gamma_{W_j}(q)$ also contain a unique BNL, and the angle between them
coincides with the angle between $\Gamma_{W_i}(p)$ and $\Gamma_{W_j}(p)$.
\end{lem}
\proof
Fix $\delta>0$ and let $M_\delta=\{x\in M^n: |\Scal_M(x)|<\delta\}$.
Then $p\in\pw\subset V\subset M_\delta$. Choosing $U_p$ and~$S$ as in
\lref{surface} we can study $\Gamma$ in $U_p$ in terms of its
intersection with~$S$, and in the following drop $S$ for clarity. In
addition, assume that $\pm\delta$ are regular values of $\Scal_M$
restricted to $S$, and observe that $M_\delta\cap U_p$ is an open
neighborhood of $\pw_i\cap\pw_j$. We denote by $r$ either $i$ or $j$.

Since $\pm\delta$ are regular values, the set
$\{|\Scal_M|\ge\delta\}\cap U_p$ is contained in the union of finitely
many closed disjoint 2-discs (or half disks)~$D_\ell$. If we remove from
$W_r$ those discs $D_\ell$ which are contained in it, we obtain the open
set $W_r'\subset W_r$, which is connected since~$W_r$~is.
Let $\mu_r\in\Gamma_{W_r}(q)$ be a BNL and set $\nu_r=\Gamma_{W_r}(p)$,
which contains only one element by \lref{only1}. Choose two sequences
$p_{r,k},q_{r,k}\in W_r'$ such that $p_{r,k}\to p$, $q_{r,k}\to q$, with
$\Gamma(q_{r,k})\to \mu_r$ and $\Gamma(p_{r,k})\to\nu_r$. Choose smooth
simple curves $\alpha_{r,k}\subset W_r'$ joining $p_{r,k}$ to $q_{r,k}$,
and let $\beta_{r,k}=\overline{p_{r,k}p}$ and
$\gamma_{r,k}=\overline{q_{r,k}q}$, which, since $\partial M_\delta$ has
positive distance to $p$ and $q$, we can assume to lie in $M_\delta$ for
$k$ sufficiently large. Thus we get two curves
$\varphi_{r,k}=\beta_{r,k}^{-1}* \alpha_{r,k}*\gamma_{r,k}$ from $p$ to
$q$, and hence a closed curve
$\varphi_k:=\varphi_{j,k}^{-1}*\varphi_{i,k} \subset M_\delta\cap U_p$,
which we can also assume to be simple. By part $(c)$ in \lref{surface},
$\varphi_k$ bounds a 2-disk $D\subset U_p$.

We claim that $\alpha_{r,k}\subset M_\delta\cap W_r$ can be modified in
such a way that $D\subset M_\delta\cap U_p$. First observe that any
closed disc $D_\ell\subset D$ as above must be contained in either $W_i$
or $W_j$ since $\partial D_\ell\cap \varphi_k=\emptyset$ and no component
has its closure contained in $U_p$. For each $D_\ell\subset W_r$, by
means of a smooth curve $\phi_\ell\subset W_r'\cap D$ connecting the
boundary of $D_\ell$ with a point $y_\ell$ in~$\alpha_{r,k}$ we can
contour $D_\ell$ from the interior of $D$ by following $\alpha_{r,k}$ up
to $y_\ell$, $\phi_\ell*\partial D_\ell*\phi_\ell^{-1}$, and the
remaining part of $\alpha_{r,k}$. We can repeat this procedure for each
$D_\ell$ and can also arrange this in such a way that all curves
$\phi_\ell$ are disjoint. Observe that this new curve, that we still call
$\alpha_{r,k}$, is contained in $W_r\cap M_\delta$, and the claim is
proved (see Figure 6).

\begin{figure}[!ht]
\centering
\includegraphics[width=0.5\textwidth]{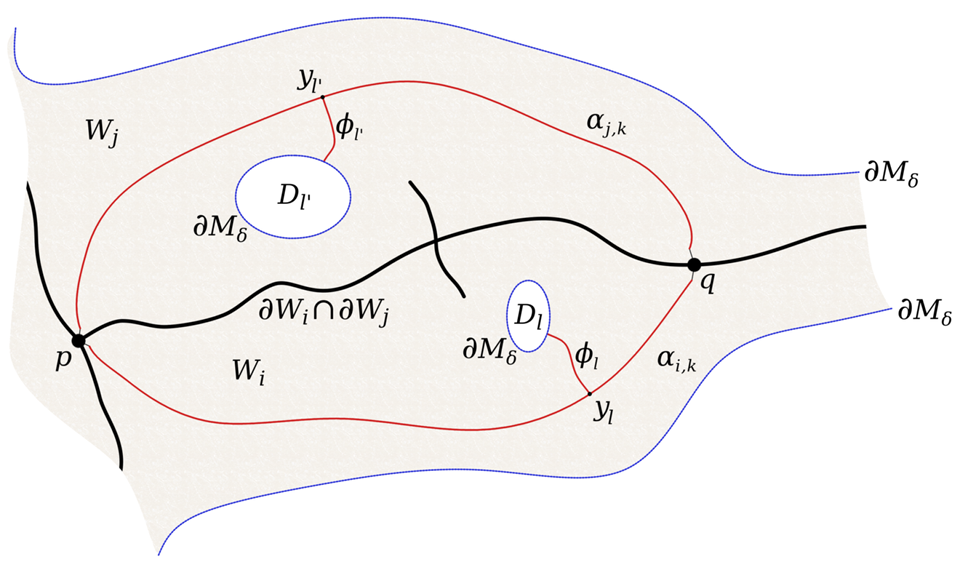}
\caption{\small A neighborhood of $\pw_i\cap\pw_j$ in $S$ with the shaded area representing $M_\delta$}
\end{figure}

As $k\to\infty$, the parallel transport of $\nu_r$ along $\varphi_{r,k}$
approaches $\mu_r$ since $\Gamma$ is parallel in~$W_r$. By
\lref{integral}, the angle between the parallel transport of $\nu_i$
along $\varphi_{i,k}$ and along $\varphi_{j,k}$ can be bounded by
$(n-1)\delta Area(S)$. Since the angle between $\nu_i$ and $\nu_j$ and
their parallel transport along $\nu_j$ is the same, the claim follows by
taking $\delta\to 0$.

Finally, assume that there are two BNL's in $\Gamma_{W_r}(q)$.
We can repeat the above argument with curves lying only in $W_r$
since we did not assume that $i\neq j$, and it follows from \lref{only1}
that the angle between them is $0$.
\qed
\begin{lem}\label{nc}
The distribution $\Gamma$ does not extend continuously to any
neighborhood of any $p\in\pw$.
\end{lem}
\proof
Suppose that $\Gamma$ extends continuously to $B_\delta(p)$. Let
$\hat\Gamma$ be the smooth distribution in $B_\delta(p)$ obtained by
parallel transporting $\Gamma(p)$ along geodesics emanating from $p$.
Choosing a surface $S$ centered at $p$ as in \lref{surface}, and
contained in $B_\delta(p)$, we first claim that on $S$ the
distribution~$\Gamma$ agrees with $\hat\Gamma$. To see this, let
$\alpha\subset S$ be a geodesic stating at $p$, and consider the angle
function $\alpha(t)$ between $\Gamma(\gamma(t))$ and
$\hat\Gamma(\gamma(t))$. An argument similar to the one in the proof of
\lref{constant} shows that $\alpha$ is locally constant along the
finitely many (not necessarily connected) sets
$\gamma\cap\overline{W_i}$. Since by hypothesis $\alpha$ is continuous
with $\alpha(0)=0$, we conclude that $\alpha=0$, as desired. Since the
leaves of both $ \Gamma$ and $\hat \Gamma$ intersect $S$ transversely,
they must agree in a neighborhood of $S$ on which $\Gamma$ is thus
smooth. By \lref{4}, this property also holds on the union of all
complete leaves going through $S$, which contradicts the maximality
of $W$.
\qed

\begin{lem}\label{noc}
For every $p\in\pw$ we have that $\, 2\leq\#\Gamma(p)\leq m(p)$.
\end{lem}
\proof
First, observe that \lref{constant} for $i=j$ shows that
$\Gamma_{W_i}(q)$ contains a unique BNL for all $q\in\pw_i\cap S$.
We claim that this implies that
$\pw\cap S = \bigcup_{i\neq j}(\pw_i\cap\pw_j\cap S)$. If not, there
exists a local component $W_i$, $q\in \pw_i\cap S$ and a neighborhood $U$
of $q$ such that $U\subset \overline{W_i}$ and
$U\cap \overline{W_j}=\emptyset$ for $j\ne i$ (see Figure 5). But then
for every $r\in U\cap\pw\cap S\subset U\cap\pw_i\cap S$ we have that
$\Gamma_{W_i}(r)$ contains a unique BNL. By \lref{4} this is then also
the case for any $r\in U\cap\pw$, i.e. $\Gamma$ is continuous in $U$,
which contradicts \lref{nc}.

We now show that $\#\Gamma(p)\ge 2$. So assume that $\Gamma(p)$ has a
unique element, i.e. $\Gamma_{W_i}(p)=\Gamma_{W_j}(p)$ for all $i\ne j$.
Then \lref{constant} implies that the same is true for any
$q\in\pw_i\cap\pw_j\cap S$ and hence by the above for any $q\in\pw\cap
S$. Thus $\Gamma$ is continuous in $U_p$, which again contradicts
\lref{nc}. The second inequality follows from \lref{only1}.
\qed

\begin{lem}\label{locconv}
There exists $\delta=\delta(p)>0$ such that $W_i\cap B_\delta(p)$ is
convex for all $i$.
\end{lem}
\proof
Let ${\e'}>0$ such that $B_{\e'}(p)\subset U_p\subset B_{\e}(p)$, and
let $\delta=\delta({\e'},p)$ as in \lref{e}.
Take points $q,q'\subset W_i\cap B_{\delta}(p)$ for which the
minimizing geodesic segment $\overline{q q'}$ is not contained in
$W_i\cap B_{\delta}(p)$. Take a curve
$\alpha\subset W_i\cap B_{{\e'}}(p)$ joining $q$ with $q'$, and set
$s:= \sup\{r:\sigma_t\subset W_i,\,\forall 0\leq t<r\}$,
where $\sigma_t=\overline{q\alpha(t)}$.
Since $\overline{q\alpha(s)}\subset \overline{W_i}\cap B_{{\e'}}(p)$,
we have that $\overline{q\alpha(s)}\cap\pw_i\not=\emptyset$.
We claim that $m(x)=1$ for all $x\in \overline{q\alpha(s)}\cap\pw_i$,
which contradicts the first inequality in \lref{noc}.

To prove the claim, take $\mu\in\Gamma(x)$ and consider for each $0<t<s$
the flat totally geodesic completely ruled hypersurface
$H_t:=\cup_{0<r<1} \Gamma(\sigma_t(r))\subset W$ with limit
$H:=\lim_{t\to s} H_t$. If $H$ intersects $\mu$ transversally, $H_t$
would also for $t$ close to $s$, which is a contradiction since
$\mu\subset\pw$. Therefore, $T_x\mu$ is a hyperplane contained in $T_xH$.
Since $H$ is foliated by complete flat hypersurfaces parallel to
$\Gamma(q)$ along $\sigma_s$ and $\mu\subset H$ is also a complete flat
hypersurface, it follows that $\mu$ is parallel to $\Gamma(q)$ along
$\sigma_s$ as well. Thus $\mu$ is unique and hence $m(x)=1$.
\qed
\vspace{1.5ex}

We now come to the main result about the local structure of $\pw$.

\begin{lem}\label{int1}
The set $F_{ij}:=\partial W_i\cap\partial W_j\subset\partial W$ is convex
for all $i,j$, and along every geodesic in $F_{ij}$ the two families of
BNL's induced by $W_i$ and $W_j$ are parallel.
\end{lem}
\proof
Take two points $q,r\in F_{ij}$ and, for $k=i,j$, sequences
$q_{k,n},r_{k,n}\in W_k$ such that $q_{k,n}\to q$, $r_{k,n}\to r$. By
convexity,
$\overline{q_{k,n}r_{k,n}}\subset W_k$ and since both converge to
$\overline{qr}$, it follows that $\overline{qr}\subset F_{ij}$.
For the second assertion, simply observe that the parallel transport
along $\overline{qr}$ of the BNL's agrees with the limits of the
parallel transport along $\overline{q_{k,n}r_{k,n}}$.
\qed
\vspace{1.5ex}

We are finally in a position to prove \tref{m2}, which follows
from \tref{m1} and the following.

\begin{theorem}\label{mainfinite}
Let $M^n$ be a complete Riemannian manifold with a parallel rank $n-2$
distribution defined in a dense, locally finite and maximal open set $W$,
whose leaves are complete and flat. Then $M^n\setminus W$ is a disjoint
union of complete flat totally geodesic embedded hypersurfaces.
If, in addition, $M^n$ has finite volume, then these hypersurfaces are
compact and $M^n$ is a \ggm\!\!.
\end{theorem}
\proof
Consider $F_{ij}$ as in \lref{int1} with
$\Gamma_{W_i}(p)\neq\Gamma_{W_j}(p)$ which exists by \lref{noc}. Then for
$r=i,j$, each point in the interior $F_{ij}^\circ$ of $F_{ij}$ is
contained in a unique complete BNL
of $W_r$, and we denote by $S_r\subset\pw_r$ the union of such BNL's with
$S_i\cap S_j\supset F_{ij}^\circ$. Observe in addition that $S_r$ is a
smooth flat totally geodesic hypersurface, and completely ruled since, as
seen in the proof of \lref{int1}, it arises as a limit of
$H_n:=\cup_{0<t<\e} \Gamma(\sigma_n(t))\subset W_r$ for a sequence of
geodesic segments $\sigma_n\subset W_r$.

We now study the local connected components based at a point $q\in
S_i\cap\partial S_j\subset\partial F_{ij}^\circ$. Clearly, $W_i$ is one
of those, with $S_i$ smooth at $q$ and $W_i$ lying (locally) on one side
of~$S_i$. Let $W'_1$, $W'_2$ be any two other local components at $q$
with BNL's $\mu_s\in \Gamma_{W'_s}(q)$ for $s=1,2$. Observe first that
$\mu_s$ cannot be transversal to $S_i$ since otherwise leaves of $\Gamma$
in $W_i$ and~$W_s'$ would intersect. Hence $\mu_1$ and $\mu_2$ are
tangent to $S_i$, which implies that $\mu_1=\mu_2$. Indeed, otherwise
$\mu_1$ and $\mu_2$ would intersect transversally in $S_i$ by dimension
reasons, and then near~$q$ the leaves of $W'_1$ and $W'_2$ would again
intersect since they are both locally on the same side of the
hypersurface $S_i$. Therefore, all \lccs at~$q$, apart from $W_i$, share
the same BNL and thus $\#\Gamma(q)=2$. Using a surface $S$
at~$q$ and $U_q\subset S$ as in \lref{surface}, we see that $W_i\cap S$
is a half disc with boundary a smooth geodesic containing~$q$ in its
interior. For all remaining local connected components $W_s'$ at~$q$,
\lref{int1} implies that the intersections $\pw_1'\cap\pw_2'\cap S$ are
geodesics with endpoints at $q$. Since
$\Gamma_{W_1'}(q)=\Gamma_{W_2'}(q)$, \lref{constant} implies that
$\Gamma_{W_1'}(r)=\Gamma_{W_2'}(r)$ for all $r\in \pw_1'\cap\pw_2'\cap S$
and hence by \lref{4} also in a neighborhood of $S$. Thus by \lref{nc}
there can be only one such component, i.e. $m(q)=2$, and hence $W_j$ is
the second component at $q$. But then $S_j$ extends past~$q$ and hence
$F_{ij}$ is a complete flat hypersurface containing $p$ in its interior.
We conclude that the number of \lccs at $\pw$ is 2 everywhere and $\pw$
is a disjoint union of complete flat totally geodesic embedded
hypersurfaces.

\vspace{0.5ex}

In order to prove the last assertion of the theorem, let $C^n=\cyl$
be one of the twisted cylinders with finite volume.
A component of its boundary has the form $H=(\gamma\times\R^{n-2})/G'$,
where $\gamma\subset\su$ is a complete boundary geodesic and
$G'\subset G$ the normal subgroup that preserves $H$. We first assume
that $G'$ acts nontrivially on $\gamma$. By taking a two-fold cover of
$C^n$ if necessary, we can assume that there are no elements of $G'$
which act as a reflection on $\gamma$, and hence $G'$ contains an element
which acts by translation. This implies that there exists a uniform $\e$
tubular neighborhood $B_\e(\gamma)\subset\su$ of the infinite geodesic
$\gamma$. On $B_\e(\gamma)$ the metric is $C^\infty$-close to the product
metric on $[0,\e)\times \gamma$ since the curvature of $\su$ vanishes to
infinite order along $\gamma$. Hence, the $\e$-tubular neighborhood of
$H$ in $C^n$ is given by
$[0,\e)\times(\gamma\times\R^{n-2})/G'=[0,\e)\times H$, with a metric
also $C^\infty$-close to a product metric. Since $C^n$ has finite volume,
so does $H$. But one easily sees that a flat manifold of finite volume is
compact.

On the other hand, if $G'$ acts trivially on $\gamma$, then
$H=\gamma\times F^{n-2}$ with BNL $F^{n-2} = \R^{n-2}/G'$. If
$B$ is a small ball near $\gamma$ such that its translates under $G'$ are
disjoint from $B$, then the projection of $ B\times \R^{n-2}$ to $C^n$ is
isometric to $B\times F^{n-2}$. Thus $F^{n-2}$ again has finite volume
and is hence compact. This implies the last assertion of the theorem
since $H$ is the boundary of another finite volume twisted cylinder that
induces different BNL's on $H$.
\qed

\vspace{1.5ex}

\begin{rem*}
One of the difficulties in proving Conjecture 1 in the Introduction is
that one needs to exclude the following situation when local finiteness
fails. Let $W'$ be a concave local connected component at $p$ with $\pw'$
consisting of two smooth hypersurfaces meeting along their common
boundary BNL at $p$. Then one needs to show that the complement of $W'$
near~$p$ cannot be densely filled with infinitely many disjoint twisted
cylinders whose diameters go to $0$ as they approach $p$.
\end{rem*}


\begin{thebibliography}{9999l} 

\bibitem[AM]{am} A. Aazami and C. Melby-Thompson,
{\it On the principal Ricci curvatures of a Riemannian 3-manifold}.
Preprint 2005, arxive:1508.02667v2.

\bibitem[BS]{sb} R. Bettiol and B. Schmidt,
{\it Three-manifolds with many flat planes},
Preprint 2014, arXiv:1407.4165.

\bibitem[BKV]{BKV} E. Boeckx, O. Kowalski and L. Vanhecke,
{\it Riemannian manifolds of conullity two}.
World Scientific, 1996.

\bibitem[Ca]{ca} E. Cartan,
{\it La d\'eformation des hypersurfaces
dans l'espace euclidien r\'eel a $n$ dimensions}.
Bull. Soc. Math. France {\bf 44} (1916), 65--99.

\bibitem[CG]{cg}
J. Cheeger and M. Gomov, {\it Collapsing Riemannian manifolds while keeping
their curvature bound I}, J. Dif. Geom. {\bf 23} (1986), 309-364.

\bibitem[DF]{df} M. Dajczer and L. Florit,
{\it Genuine rigidity of Euclidean submanifolds in codimension two}.
Geom. Dedicata {\bf 106} (2004), 195--210.

\bibitem[DFT]{dft} M. Dajczer, L. Florit and R. Tojeiro,
{\it On deformable hypersurfaces in space forms}.
Ann. Mat. Pura ed Appl. {\bf 147} (1998), 361--390.

\bibitem[FF]{ff} L. Florit and G. Freitas,
{\it Classification of codimension two deformations
of rank two Riemannian manifolds}.
To appear in Comm. Anal. Geom.

\bibitem[FZ1]{fz} L. Florit and W. Ziller,
{\it Nonnegatively curved Euclidean submanifolds in codimension two}.
Comm. Math. Helv. {\bf 91} (2016), no. 4, 629--651.

\bibitem[FZ2]{fz1} L. Florit and W. Ziller,
{\it Geometric graph manifolds with nonnegative scalar curvature}.
Preprint, 2017. ArXiV:
\href{https://arxiv.org/abs/1705.04208}{1705.04208}.

\bibitem[Gr]{g} M. Gromov,
{\it Manifolds of negative curvature}.
J. Differ. Geom. {\bf 13} (1978), 223--230.

\bibitem[Ma]{ma} R. Maltz,
{\it The nullity spaces of curvature-like tensors}.
J. Diff. Geom. {\bf 7} (1972), 519--523.

\bibitem[No]{no} K. Nomizu,
{\it On hypersurfaces satisfying a certain condition on the curvature
tensor}.
Tohoku Math. J. {\bf 20} (1968), 46--59.

\bibitem[PR]{pr} R. Ponge and H. Reckziegel,
{\it Twisted products in pseudo-Riemannian geometry}.
Geom. Dedicata {\bf 48} (1993), no. 1, 15--25.

\bibitem[SW]{sw} B. Schmidt and J. Wolfson,
{\it Three manifolds with constant vector curvature}.
Indiana Univ. Math. J. {\bf 63} (2014), 1757--1783.

\bibitem[Sb]{sbr} V. Sbrana, {\it Sulla variet\'a ad $n-1$ dimensioni
deformabili nello spazio euclideo ad $n$ dimensioni}.
Rend. Circ. Mat. Palermo {\bf 27} (1909), 1--45.

\bibitem[Sch]{sc} V. Schroeder,
{\it Rigidity of Nonpositively Curved Graphmanifolds}.
Math. Ann. {\bf 274} (1986), 19--26.

\bibitem[Se1]{se1} K. Sekigawa,
{\it On the Riemannian manifolds of the form $B \times_f F$}.
Kodai Math. Sem. Rep. {\bf 26} (1974), 343--347.

\bibitem[Se2]{se2} K. Sekigawa,
{\it On some 3-dimensional complete Riemannian manifolds satisfying $R(X,Y)\cdot R=0$}.
Tohoku Math. J. {\bf 27} (1975), 561--568.

\bibitem[Se3]{se3} K. Sekigawa,
{\it On some 4-dimensional Riemannian manifolds satisfying $R(X,Y)\cdot R=0$}.
Hokkaido Math. J. {\bf 6} (1977), 216--229.

\bibitem[Si]{si} I. Singer,
{\it Infinitesimally homogeneous spaces}.
Comm. Pure Appl. Math. {\bf 13} (1960) 685--697.

\bibitem[Sz]{sz} Z.I. Szab\'o,
{\it Structure theorems on Riemannian spaces satisfying
$R(X,Y)\cdot R=0$, II, Global versions}.
Geom. Dedicata {\bf 19} (1985), 65--108.

\bibitem[Wa]{w} F. Waldhausen,
{\it Eine Klasse von 3-dimensionalen Mannigfaltigkeiten II}.
Invent. Math. {\bf 4} (1967), 87--117.

\end{thebibliography}
\end{document}